\documentclass[a4paper,11pt]{article}

\usepackage[headinclude]{typearea}
\usepackage[english]{babel}           
\usepackage[utf8]{inputenc}
\usepackage[T1]{fontenc}
\usepackage{scrtime}
\usepackage{amsmath,amsthm,amsfonts,amssymb}
\usepackage {color}                                  
\usepackage {pifont}                                 
\usepackage {multicol}                               
\usepackage {times}                                  
\usepackage[numbers]{natbib}                                                     
\usepackage{helvet}      
\usepackage {courier}                                
\usepackage {graphicx}                               
\usepackage {array}                                  
\usepackage {longtable}                              
\usepackage {makeidx}                                
\makeindex
\usepackage{fancyvrb}
\usepackage{enumerate} 
\usepackage{ifpdf}
\usepackage{caption}

\usepackage{setspace}

\usepackage{marvosym}

\theoremstyle{plain}
\newtheorem{theorem}{Theorem}[section]

\newtheorem{lemma}[theorem]{Lemma}
\newtheorem{proposition}[theorem]{Proposition}

\theoremstyle{definition}
\newtheorem{assumption}[theorem]{Assumption}
\newtheorem{definition}[theorem]{Definition}
\newtheorem{example}[theorem]{Example}

\newcommand{\E}{{\mathbb{E}}}

\renewcommand{\P}{{\mathbb{P}}}

\newcommand{\R}{{\mathbb{R}}}

\newcommand{\sign}{\operatorname{sign}}

\definecolor{darkgreen}{rgb}{0,0.5,0}

\newcommand{\EM}{\phi}
\newcommand{\LS}{\Phi}

\renewcommand{\P}{{\mathbb P}}

\newcommand{\cE}{{\cal E}}
\newcommand{\cF}{{\cal F}}

\newcommand{\cL}{{\cal L}}

\newcommand{\be}{\begin{equation}}
\newcommand{\ee}{\end{equation}}
\newcommand{\bea}{\begin{eqnarray}}
\newcommand{\eea}{\end{eqnarray}}
\newcommand{\beast}{\begin{eqnarray*}}
\newcommand{\eeast}{\end{eqnarray*}}
\newcommand{\bproof}{\begin{proof}}
\newcommand{\eproof}{\end{proof}}

\title{A Numerical Method for SDEs with Discontinuous Drift}

\author{Gunther Leobacher and Michaela Sz\"olgyenyi\,
\thanks{The authors thank Evelyn Buckwar (Johannes Kepler University Linz) for valuable discussions. G. Leobacher and M. Sz\"olgyenyi are supported by the Austrian Science Fund (FWF): Project F5508-N26, which is part of the Special Research Program "Quasi-Monte Carlo Methods: Theory and Applications".}}

\begin{document}

\date{Preprint, January 2015}

\maketitle

\begin{abstract}
In this paper we introduce a transformation technique, which can on the one hand be used to prove existence and uniqueness for a class of SDEs with discontinuous drift coefficient.
One the other hand we present a numerical method based on transforming the Euler-Maruyama scheme for such a class of SDEs.
We prove convergence of order $1/2$. Finally, we present numerical examples.\\
 
\end{abstract}

\centerline{\underline{\hspace*{16cm}}}

\noindent G. Leobacher \\
Department of Financial Mathematics, Johannes Kepler University Linz, 4040 Linz, Austria\\

\noindent M. Sz\"olgyenyi \Letter \\
Department of Financial Mathematics, Johannes Kepler University Linz, 4040 Linz, Austria\\
michaela.szoelgyenyi@jku.at

\centerline{\underline{\hspace*{16cm}}}

\section{Introduction}
\label{sec:Introduction}

We consider a time-homogeneous stochastic differential equation (SDE)
\begin{align}\label{eq:SDEintro}
dX_t &= \mu(X_t) \, dt + \sigma (X_t) dW_t\,,
\end{align}
where $\sigma$ is a Lipschitz, $\R$-valued function and 
where $\mu$ is an $\R$-valued function that is allowed to 
have discontinuities.

It is well-known that the problem of existence and uniqueness of a 
solution to \eqref{eq:SDEintro} is readily settled by Picard iteration, if $\mu$ is Lipschitz, too.
In the case where the diffusion coefficient is bounded, Lipschitz, and (partly) uniformly elliptic, and the drift coefficient $\mu$ is only bounded and measurable,
the pioneering work by \citet{zvonkin1974} and \citet{veretennikov1981,veretennikov1984} yields existence and uniqueness of the solution.
There the result is achieved by applying a transform 
that removes the drift. This transform can in principle be computed by solving a
non-degenerate elliptic partial differential equation.
In the one-dimensional case this reduces to solving an ordinary differential equation.
Nevertheless, from the point of view of numerical treatment of \eqref{eq:SDEintro} 
this transformation method is impractical. 

In \citet{sz14} one can find an existence and uniqueness result for multi-dimensional SDEs, for the case where the drift is allowed to be
discontinuous at a hyperplane, or at a hypersurface, but is well behaved everywhere
else. Instead of removing the whole drift by a transformation, only 
the discontinuity is removed by using a different
transformation method. For computing this transform one only needs to  solve a
parametrized family of ODEs which can be done by iterated integration. This
transform is therefore much less costly from a computational point of view.\\

In this paper, we prove another existence and uniqueness result for
\eqref{eq:SDEintro} under conditions weaker than those in \citet{sz14}.
The transform constructed here is explicit and thus gives rise to  a
numerical method that does not require
solving any (partial) differential equation.\\

In setups with non-globally Lipschitz drift coefficient, various authors have
studied convergence of numerical schemes.  \citet{berkaoui2004} proves strong
convergence of the Euler-Maruyama scheme for $C^1$ drift.
\citet{hutzenthaler2012} present an explicit numerical method for which they
are able to prove strong convergence in case of an SDE with non-globally
Lipschitz coefficient.  \citet{gyongy1998} proves almost sure convergence of
the Euler-Maruyama scheme in the case where the drift satisfies  a monotonicity
condition.  \citet{halidias2008} show that the Euler-Maruyama scheme converges
strongly in case of a discontinuous monotone drift coefficient, e.g., in the
case where the drift is a Heaviside function.  
\citet{kohatsu2013} show -- in case of a
discontinuous drift -- weak convergence of a method where they first regularize
the drift and then apply the Euler-Maruyama scheme.
\citet{etore2013,etore2014} present an exact simulation algorithm for SDEs with
a bounded drift coefficient that has a discontinuity in one point, but is differentiable everywhere else.\\

In contrast to that, we allow the drift to have a finite number of jumps,
but require it to be Lipschitz otherwise.
The transformation method and the scheme based on it are the main contributions of this paper.
The transformation itself is presented in a constructive way.
Moreover, it is chosen such that all the involved functions can be computed efficiently.\\

The paper is organized as follows.
In Section \ref{sec:ex} we introduce a transformation method different to those from \cite{sz14, zvonkin1974} and show that this also leads to an existence and uniqueness result.
In Section \ref{sec:Num} we present the numerical method which is based on the transformation introduced earlier and we prove convergence of strong order $1/2$.
Finally, in Section \ref{sec:Example} we present numerical examples.

\section{Existence and uniqueness}
\label{sec:ex}

Let $(\cE,\cF,(\cF)_{t\ge0},\P)$ be a filtered probability space carrying a
standard Brownian motion $W=(W_t)_{t\ge0}$. Let $\mu, \sigma: \R\to \R$, be measurable functions.  
We study the time-homogeneous stochastic differential equation (SDE)
\begin{align}
\label{eq:SDE}dX_t &= \mu(X_t) \, dt + \sigma (X_t) dW_t\,, \qquad X_0=x\,.
\end{align}

The function $\mu$ is allowed to be discontinuous. However, the form of the discontinuities is a special one: 
we allow only discontinuities in a finite number $m$ of distinct points $\xi_1<\ldots<\xi_m$ and we assume that $\mu$ is Lipschitz otherwise.

\begin{definition}
Let $I\subseteq\R$ be an interval. We say a function $f:I\longrightarrow\R$
is piecewise Lipschitz if there are finitely many points $\xi_1<\ldots<\xi_m\in
I$ such that $f$ is Lipschitz on each of the intervals $(-\infty,\xi_1)\cap I,
(\xi_m,\infty)\cap I$ and $(\xi_k,\xi_{k+1}),\,k=1,\ldots,m$.
\end{definition}

We have not assumed anything about the behaviour of $f$ at 
$\xi_1,\ldots,\xi_m$. However, we have the following elementary lemma:

\begin{lemma}\label{lem:one-sided}
For a piecewise Lipschitz function $f$ the one-sided limits always 
exist. 
\end{lemma}

\begin{proof}
By symmetry it is enough to show that the left limit exists in every point.
For given $x\in I$ consider a non-decreasing sequence $(x_n)$ with
$x_n\rightarrow x$. 
According to our assumption there is some $L$ and some $n_0$ such that
$|f(x_n)-f(x_k)|\le L|x_n-x_k|$ for all $n,k \ge n_0$. Now 
\begin{align*}
\sum_{n=n_0}^\infty |f(x_{n+1})-f(x_n)|\le \sum_{n=n_0}^\infty L|x_{n+1}-x_n|
=L \sum_{n=n_0}^\infty (x_{n+1}-x_n)=L(x-x_{n_0})<\infty
\end{align*}
such that $\sum_{n=n_0}^\infty (f(x_{n+1})-f(x_n))$ converges absolutely and the limit of this series is $\lim_{n\rightarrow \infty}f(x_n)-f(x_{n_0})$. 
\end{proof}

This will enable us to transform the SDE into one with Lipschitz 
coefficients. The transform presented here is similar to the one in
\cite{sz14} but much simpler to compute.

\begin{assumption}\label{ass:hyper-all}
We assume the following for the coefficients of \eqref{eq:SDE}:
 \begin{enumerate}
 \renewcommand{\theenumi}{\roman{enumi}}
 \renewcommand{\labelenumi}{(\theenumi)}
  \item\label{ass:hyper-all-mu1}
  $\mu$ is piecewise Lipschitz;
 \item\label{ass:hyper-all-sigma3} 
$\sigma$ is globally Lipschitz;
  \item\label{ass:hyper-all-sigma2} There exists $\bar c>0$ such that 
$\sigma^2(\xi_i) \ge \bar c $ for  all $i=1,\ldots, m$. 
 \end{enumerate}
\end{assumption}

We want to stress that those assumptions are satisfied by many practical examples.
A classical one is the one-dimensional process $X$
satisfying $X_0=x$ and
\begin{align*}
d X_t=-\sign(X_t)dt + dW_t\,.
\end{align*}

Our assumptions on  $\mu$  enable us to remove the discontinuity from the drift
by a suitable transformation of $X$.
The transformation $g$ is chosen such that $g''$ is piecewise linear and is 
non-zero only on environments of the discontinuities.
This yields a piecewise quadratic function $g'$, which is constantly 1 except
on environments of the discontinuities.  Furthermore, we can guarantee
boundedness of the derivatives of $g$.

\begin{proposition}\label{prop:exgh}
Let $\xi_1,\ldots, \xi_m\in\R$ with $\xi_1<\ldots<\xi_m$ and let 
$\alpha_1,\ldots,\alpha_m,\beta_1,\ldots,\beta_m\in \R$. 
Let $0<\varkappa<1$ be fixed.

There exist functions $g,h:\R\longrightarrow \R$ such that
\begin{enumerate}
\item $g(h(z))=z\,$ for all $z\in \R$ and $h(g(x))=x\,$ for all $x\in \R$;
\item  $g,h\in C^1(\R)$;
\item $\|g'-1\|_\infty\le\varkappa$ and $\|h'-1\|_\infty\le\varkappa$;
\item $\sup_{x\in\R}|g(x)-x|\le \frac{\varkappa}{2} \max\left(1,\max_{1\le k\le m} (\xi_{k+1}-\xi_{k})\right)$ and $g(x)=x\,$ for $|x|$ sufficiently large;
\item $g'',h''$ are piecewise continuous and bounded on $\R$, 
both  one-sided limits of $g''$ at $\xi_k$ exist, and both one-sided limits
of $h''$ at $g(\xi_k)$ exist for $k=1,\ldots,m$;
\item $g''(\xi_k+)=\alpha_k$ and $g''(\xi_k-)=\beta_k$, $k=1,\ldots,m$.
\end{enumerate}
\end{proposition}

\begin{proof}
Let $\xi_0=\xi_1-1$,  $\xi_{m+1}=\xi_m+1$, $\alpha_0=\beta_0=\alpha_{m+1}=\beta_{m+1}=0$.
We construct $g''$ on the intervals $[\xi_{k-1},\xi_k]$ for $k=1,\ldots,m+1$. 

Choose $c_k\in[\xi_{k-1},\xi_{k-1}+\frac{1}{4}(\xi_{k-1}+\xi_k))$ 
($c_k$ will depend on $\varkappa$) and let $g''$ on $[\xi_{k-1},\frac{1}{2}(\xi_{k-1}+\xi_k))$ be
\begin{align*}
 g''(x)=\begin{cases}
\alpha-\frac{3\alpha(x-\xi)}{c-\xi}\,,  & x\in (\xi,\frac{1}{2}(\xi+c)]\\
 \alpha\frac{x- c}{ c-\xi}\,, & x\in (\frac{1}{2}(\xi+ c),c]\\
 -\frac{8\alpha (x-c)}{3( c-\xi)}\,, & x \in (c, c+\frac{ c-\xi}{4}]\\
 \frac{4\alpha(\xi-c+2(x-c))}{3(c-\xi)}\,, & x \in ( c+\frac{ c-\xi}{4}, c+\frac{3( c-\xi)}{4}]\\
 -\frac{8 \alpha (\xi-2 c+x)}{3(x-\xi)}\,, & x \in ( c+\frac{3( c-\xi)}{4},2 c-\xi]\\
 0\,, & x\in (2 c-\xi,\frac{1}{2}(\xi+\xi_k)]\,,
 \end{cases}
\end{align*}
where we write $\xi=\xi_{k-1}$, $\alpha=\alpha_{k-1}$, and $c = c_k$ for brevity.
Note that
\begin{align*}
 \int_{\xi_{k-1}}^{\frac{\xi_{k-1}+\xi_k}{2}} g''(x) dx =0
 \qquad \mbox{ and } \qquad
 \int_{\xi_{k-1}}^{\frac{\xi_{k-1}+\xi_k}{2}} \int_{\xi_{k-1}}^x g''(t)dt\, dx =0\,,
\end{align*}
and that we choose $ c_k$ such that 
\[
\max_{\xi_{k-1} \le x \le \frac{\xi_{k-1}+\xi_k}{2}} \left|\int_{\xi_{k-1}}^x g''(t) dt \right| \le \varkappa/(1+\varkappa)\,.
\]
On $[\frac{1}{2}(\xi_{k-1}+\xi_k),\xi_{k}]$ we define $g''$ analog with $g''(\xi_k)=\beta_k$.
Further define $g''(x)=0$ for $x< \xi_0$ and $x>\xi_{m+1}$.

Now let
\begin{align*}
 g(x)=x+\int_0^x \int_0^t g''(s)ds \, dt\,.
\end{align*}
Then it is easy to verify that $g$ satisfies items 2 -- 6.

The function $g$ is piecewise cubic with positive derivative and 
therefore has a global inverse $h$ that can be given explicitly as a piecewise
radical function. We have 
\[
|h'(z)-1|=\left|\frac{1}{g'(h(z))}-1\right|=\frac{|1-g'(h(z))|}{|g'(h(z))|}
\le \frac{\frac{\varkappa}{1+\varkappa}}{1-\frac{\varkappa}{1+\varkappa}}=\varkappa\,.
\]
\end{proof}

The proof of Proposition \ref{prop:exgh} is constructive.
So now we know how we can construct the function for the numerical approximation.
Figure \ref{fig:gxprime} shows the functions $ g'$ and  $g''$ on the right of some
$\xi$.

\begin{figure}[ht]
\begin{center}
\begin{picture}(0,0)%
\includegraphics{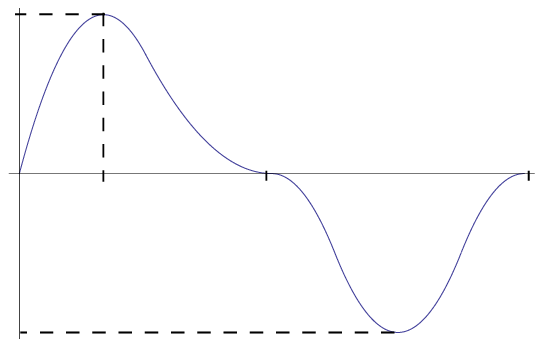}%
\end{picture}%
\setlength{\unitlength}{4144sp}%
\begingroup\makeatletter\ifx\SetFigFont\undefined%
\gdef\SetFigFont#1#2#3#4#5{%
  \reset@font\fontsize{#1}{#2pt}%
  \fontfamily{#3}\fontseries{#4}\fontshape{#5}%
  \selectfont}%
\fi\endgroup%
\begin{picture}(2832,1959)(481,-3448)
\put(496,-1726){\makebox(0,0)[b]{\smash{{\SetFigFont{6}{7.2}{\rmdefault}{\mddefault}{\updefault}{\color[rgb]{0,0,0}$1+\frac{\alpha (c-\xi)}{6}$}%
}}}}
\put(1081,-2581){\makebox(0,0)[lb]{\smash{{\SetFigFont{6}{7.2}{\rmdefault}{\mddefault}{\updefault}{\color[rgb]{0,0,0}$\frac{c+2\xi}{3}$}%
}}}}
\put(631,-2446){\makebox(0,0)[lb]{\smash{{\SetFigFont{6}{7.2}{\rmdefault}{\mddefault}{\updefault}{\color[rgb]{0,0,0}$1$}%
}}}}
\put(1936,-2536){\makebox(0,0)[b]{\smash{{\SetFigFont{6}{7.2}{\rmdefault}{\mddefault}{\updefault}{\color[rgb]{0,0,0}$c$}%
}}}}
\put(3151,-2581){\makebox(0,0)[b]{\smash{{\SetFigFont{6}{7.2}{\rmdefault}{\mddefault}{\updefault}{\color[rgb]{0,0,0}$2c-\xi$}%
}}}}
\put(811,-2536){\makebox(0,0)[lb]{\smash{{\SetFigFont{6}{7.2}{\rmdefault}{\mddefault}{\updefault}{\color[rgb]{0,0,0}$\xi$}%
}}}}
\put(541,-3211){\makebox(0,0)[b]{\smash{{\SetFigFont{6}{7.2}{\rmdefault}{\mddefault}{\updefault}{\color[rgb]{0,0,0}$1-\frac{\alpha (c-\xi)}{6}$}%
}}}}
\end{picture}%

\begin{picture}(0,0)%
\includegraphics{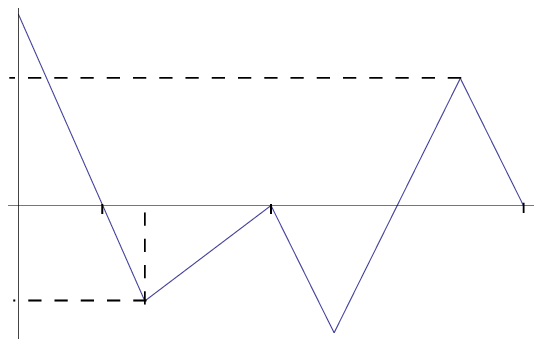}%
\end{picture}%
\setlength{\unitlength}{4144sp}%
\begingroup\makeatletter\ifx\SetFigFont\undefined%
\gdef\SetFigFont#1#2#3#4#5{%
  \reset@font\fontsize{#1}{#2pt}%
  \fontfamily{#3}\fontseries{#4}\fontshape{#5}%
  \selectfont}%
\fi\endgroup%
\begin{picture}(2784,1959)(529,-3448)
\put(656,-2608){\makebox(0,0)[lb]{\smash{{\SetFigFont{6}{7.2}{\rmdefault}{\mddefault}{\updefault}{\color[rgb]{0,0,0}$0$}%
}}}}
\put(1261,-2536){\makebox(0,0)[lb]{\smash{{\SetFigFont{6}{7.2}{\rmdefault}{\mddefault}{\updefault}{\color[rgb]{0,0,0}$\frac{\xi+c}{2}$}%
}}}}
\put(1126,-2761){\makebox(0,0)[b]{\smash{{\SetFigFont{6}{7.2}{\rmdefault}{\mddefault}{\updefault}{\color[rgb]{0,0,0}$\frac{c+2\xi}{3}$}%
}}}}
\put(676,-1726){\makebox(0,0)[lb]{\smash{{\SetFigFont{6}{7.2}{\rmdefault}{\mddefault}{\updefault}{\color[rgb]{0,0,0}$\alpha$}%
}}}}
\put(721,-2716){\makebox(0,0)[lb]{\smash{{\SetFigFont{6}{7.2}{\rmdefault}{\mddefault}{\updefault}{\color[rgb]{0,0,0}$\xi$}%
}}}}
\put(1936,-2716){\makebox(0,0)[b]{\smash{{\SetFigFont{6}{7.2}{\rmdefault}{\mddefault}{\updefault}{\color[rgb]{0,0,0}$c$}%
}}}}
\put(3106,-2671){\makebox(0,0)[b]{\smash{{\SetFigFont{6}{7.2}{\rmdefault}{\mddefault}{\updefault}{\color[rgb]{0,0,0}$2c-\xi$}%
}}}}
\put(631,-3031){\makebox(0,0)[lb]{\smash{{\SetFigFont{6}{7.2}{\rmdefault}{\mddefault}{\updefault}{\color[rgb]{0,0,0}$-\frac{\alpha}{2}$}%
}}}}
\put(631,-2041){\makebox(0,0)[lb]{\smash{{\SetFigFont{6}{7.2}{\rmdefault}{\mddefault}{\updefault}{\color[rgb]{0,0,0}$\frac{2\alpha}{3}$}%
}}}}
\end{picture}%
\end{center}
\caption{The functions $g'$ and $g''$ close to $\xi$.}\label{fig:gxprime}
\end{figure}

The function $g''$ is piecewise Lipschitz.
The function $g'$, or the constants $\alpha_1,\beta_1,\ldots,\alpha_m,\beta_m$ are chosen such that the discontinuities are removed from the
drift in a way such that the remaining term is locally Lipschitz, i.e.,
$\alpha_k=2\frac{\bar \mu_k-\mu(\xi_k+)}{\sigma^2(\xi_k)}$, $\beta_k=2\frac{\bar \mu_k-\mu(\xi_k-)}{\sigma^2(\xi_k)}$, and $\bar \mu_k=\frac{\mu(\xi_k-)+\mu(\xi_k+)}{2}$.
We define the transformed SDE by $Z=g(X)$. Then
\begin{align}\label{eq:transformed}
dZ_t &= g' dX_t + \frac{1}{2} g'' d[X]_t
= \left(\mu g'+ \frac{1}{2}\sigma^2 
g'' \right) d t + \sigma g' dW_t\,.
\end{align}

Note that $0<\varkappa<1$ was arbitrary. Therefore, if $\varkappa$ is very close
to 0, then $g$ is close to the identity.  
However, we want to stress that
there is no need to make $\varkappa$ particularly small. On the contrary, 
in general smaller 
$\varkappa$ gives rise to bigger  Lipschitz constants of the transformed
coefficients.  $\varkappa$ only needs to be
smaller than 1 to ensure that $g$ is strictly increasing. 
Note further that around each
discontinuity we build two such splines -- one on the right hand side, and one
on the left hand side of the discontinuity.  
This is done to 
reduce the Lipschitz constant of the coefficients of the transformed
equation compared to a one-sided compensation of the jumps.  

In Section \ref{sec:Example} we plot the coefficients of the transformed
SDE for Example \ref{ex:ex2}. \\

Now we are ready to prove the existence and uniqueness result, which is
obtained by applying the transformation $g$.

\begin{theorem}\label{th:existence}
Under Assumption \ref{ass:hyper-all} we have that for every $x\in\R$ there exists
a unique global strong solution $Z$ to the SDE 
\begin{align*}
dZ_t &= \tilde \mu(Z_t)dt + \tilde \sigma (Z_t) dW_t\,,\\
Z_0&= g(x)\,,
\end{align*}
where $\tilde \mu(z):= \mu(h(z))g'(h(z)) + \frac{1}{2}\sigma^2(h(z)) 
g''(h(z))$ and $\tilde \sigma (z) := \sigma(h(z)) g'(h(z))$.
Furthermore, 
$h(Z)$ is a unique global strong solution to \eqref{eq:SDE}.
\end{theorem}

For the proof we need the following elementary lemma:

\begin{lemma}\label{th:lipschitz}
Let $f:\R\longrightarrow\R$ be piecewise Lipschitz and continuous.

Then $f$ is Lipschitz on $\R$. 
\end{lemma}

\begin{proof}
There exist $a_1<\ldots<a_m$ such that 
$f$ is Lipschitz on each of the 
open intervals $(-\infty,a_1),(a_1,a_2),\ldots,(a_m,\infty)$. From 
the continuity of $f$ we conclude that $f$ is Lipschitz on each of the 
closed intervals $(-\infty,a_1],[a_1,a_2],\ldots,[a_m,\infty)$.

Let $L_0,\ldots,L_m$ denote the respective Lipschitz constants and let
$L=\max(L_0,\ldots,L_m)$.
Now let $x,y\in \R$. W.l.o.g.~$y<x$. 
If $x$ and $y$ are in the same interval, then it is obvious
that $|f(x)-f(y)|\le L |x-y|$. Otherwise, we have
$y\le a_k<...<a_{j}\le x$ with $k$ chosen minimal and $j$ chosen maximal.
\begin{align*}
|f(x)-f(y)|
&=|f(x)-f(a_j)+f(a_j)-f(a_{j-1})+\ldots+f(a_k)-f(y)|\\
&\le|f(x)-f(a_j)|+|f(a_j)-f(a_{j-1})|+\ldots+|f(a_k)-f(y)|\\
&\le L|x-a_j|+L|a_j-a_{j-1}|+\ldots+L|a_k-y|\\
&= L(x-a_j)+L(a_j-a_{j-1})+\ldots+L(a_k-y)\\
&= L(x-y)=L|x-y|\,.
\end{align*}
Thus the assertion is proven.
\end{proof}

\begin{proof}[Proof of Theorem \ref{th:existence}]
Define $\hat \mu(x):= \mu(x)g'(x) + \frac{1}{2} \sigma^2(x)
g''(x)$ and $\hat \sigma (x) := \sigma(x) g'(x)$.

Due to Assumption \ref{ass:hyper-all} (\ref{ass:hyper-all-sigma3}), $\sigma$ is
globally Lipschitz.  Furthermore, $g'$ is differentiable with bounded
derivative, which implies Lipschitz continuity.  So $\sigma$ and $g'$ are both
Lipschitz and bounded, thus $\hat \sigma=\sigma g'$ is  Lipschitz.

To show that $\hat \mu$ is Lipschitz, we first note that $g$ is
chosen in a way such that $\hat \mu$ is continuous.
We observe that for $|x|$ large it holds that $g'(x)=1$ and $g''(x)=0$.
Thus there exists $a>\max(|\xi_1|,|\xi_m|)$ such that $\hat \mu$ is
Lipschitz on $(-\infty,-a),(a,\infty)$.

Furthermore, $\hat \mu$ is Lipschitz on the intervals
$(-a,\xi_1),(\xi_1,\xi_2),\ldots,(\xi_{m-1},\xi_m),(\xi_m,a)$ 
as a sum of products of bounded
Lipschitz functions. That means that $\hat \mu$ is piecewise Lipschitz.

Hence $\hat \mu$ is  Lipschitz by Lemma \ref{th:lipschitz}.

Since $g'$ is bounded away
from 0, $h'$ is bounded and thus $h$ is Lipschitz.
Thus, also $\tilde \mu, \tilde \sigma$  are Lipschitz.\\

From the Lipschitz continuity of $\tilde \mu$ and $\tilde \sigma$ 
we get existence and uniqueness of a global strong solution to
\eqref{eq:transformed} from \cite[Theorem 3.1]{mao2007}. Furthermore, note
that It\^o's formula holds for $h$ by \cite[Problem 7.3]{karatzas1991}.
By applying It\^o's formula to $h$ we get that there exists a unique global strong solution to \eqref{eq:SDE}.
\end{proof}

\section{Numerical scheme}
\label{sec:Num}

The numerical scheme is based on the transformation $g$ introduced in Section \ref{sec:ex}.
In fact, we transform the initial value by applying $g$ and then solve the transformed SDE by applying the Euler-Maruyama method.
This converges strongly with order $1/2$ to the solution of the transformed SDE.
Then we apply $h$ to get the solution of our original SDE.\\

For any $t\ge 0$ and $\delta>0$ we denote the $n$-step
Euler-Maruyama scheme with step-size $\delta$ recursively by
\begin{align*}
\EM^1(z,t,\delta)&:=z+\tilde \mu(z)\delta 
+\tilde \sigma(z)\left(W_{t+\delta}-W_t\right) \\
\EM^{n+1}(z,t,\delta)&:=\EM^n(\EM^1(z,t,\delta),t+\delta,\delta)\,.
\end{align*}

The scheme for solving \eqref{eq:SDE} looks as follows:
\begin{align}\label{eq:LS}
\LS(x,t,T,n):=h(\EM^{n}(g(X_t),t,T/n))\,.
\end{align}

Now we are ready to prove convergence of scheme \eqref{eq:LS}.

\begin{theorem}\label{th:conv}
Let Assumptions \ref{ass:hyper-all} hold.

Then scheme \eqref{eq:LS} converges with order $\gamma=1/2$ to the solution of \eqref{eq:SDE}, i.e.,
\begin{align*}
 \E \left( \| X_T-\LS(x,0,T,n)\|^2\right)^{1/2}\le C \delta^{\gamma}\,.
 \end{align*}
\end{theorem}

\begin{proof}
\begin{align*}
 \E \left( \| X_T-\LS(x,0,T,n)\|^2\right)^{1/2}
 &= \E \left( \| h(g(X_T))-h(\EM^{n}(g(x),0,T/n))|^2\right)^{1/2}\\
&\le L_h \E\left(\|Z_{T}-\EM^{n}(g(x),0,T/n)\|^2\right)^{1/2}\,,
 \end{align*}
where $Z_T$ is the exact solution to the transformed SDE.
From \cite[Theorem 10.2.2]{kloeden1992} and as $\tilde \mu$ and $\tilde \sigma$
are Lipschitz 
we have \begin{align*}\label{eq:ESconv}
\E\left(\|Z_T-\EM^{n}(z,0,\delta)\|^2\right)\le \tilde C \delta\,.
\end{align*}

Altogether this yields
\begin{align*}
 \E \left( \| X_T-\LS(x,0,T,n)\|^2\right)^{1/2}
 \le C \delta^{1/2}\,,
 \end{align*}
where $C=L_h \sqrt{\tilde C}$, and $\delta=T/n$.

\end{proof}

\section{Examples}
\label{sec:Example}

In this section we present numerical examples, where we show the exact choice of the transformation and the transformed parameters.
Furthermore, we investigate the convergence of our method and compare it to the Euler-Maruyama method.

\begin{example} \label{ex:ex1}
First, we consider the process $X$ mentioned in the beginning and satisfying
the equation
\[
dX_t=-\sign(X_t) dt + d W_t\,.
\]
That is, $\mu(x)=-\sign(x)$ and
$\sigma(x)\equiv 1$. We choose different values of $\varkappa=1/16,1/64,1/256$.
Figure \ref{fig:skewplot} shows the estimated $\cL^2$-error between two consecutive discretizations of the transformed Euler-Maruyama method (EMT $1/\varkappa$)
in comparison to crude Euler-Maruyama (EM). This means we calculate $\log \sqrt{\hat E\left(\left(X_T^{(k)} - X_t^{(k-1)}\right)^2\right)}$ and plot it over $\log \delta ^{(k)}$,
where $X_T^{(k)}$ is the numerical approximation with stepsize $\delta=\delta^{(k)}$ and $\hat E$ is an estimator of the mean value using 1024 paths.

\begin{figure}
\begin{center}
\includegraphics[scale=0.7]{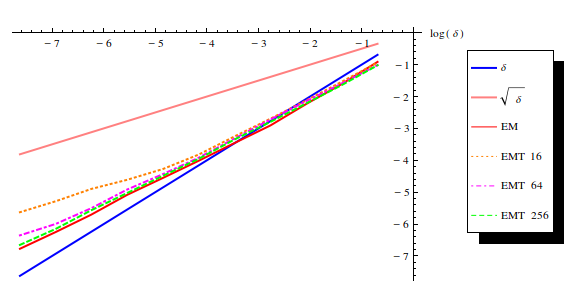}
\caption{The estimated $\cL^2$-error for different values of $\varkappa$ for Example \ref{ex:ex1}.}\label{fig:skewplot}
\end{center}
\end{figure}

We see that for this parameter choice the crude Euler-Maruyama method is even better than the transformed method.
Furthermore, we see that the estimated $\cL^2$-error for the crude Euler-Maruyama method seems to be of order $1$.
The reason for this is probably that
in case of Lipschitz coefficients Euler-Maruyama converges with strong order 1, if the diffusion parameter is constant.
However, we cannot expect strong order 1 for the transformed method, since $\tilde \sigma$ is not constant, even if $\sigma$ is constant.

\end{example}

In the next example $\mu$ has more than one discontinuity.
In addition to the estimated $\cL^2$-error between two consecutive discretizations we will show how the transformation looks like.

\begin{example}\label{ex:ex2}
Now, we consider the SDE
\[
dX_t=\mu(X_t) dt + \sigma(X_t) d W_t\,.
\]
The coefficients are chosen as follows:
\begin{align*}
 \mu (x) &=\begin{cases}
  x-2\,, & x < -1\\
  2\,,  &-1  \le x < -0.5\\
  1-x^2\,, & -0.5 \le x<0\\
  x^2\,,  &0 \le x<1\\
  -x-1\,, & x \ge 1\,,
 \end{cases}\\
 \sigma(x)&=\frac{1}{2}\left(1+\frac{1}{x^2+1}\right)\,.
\end{align*}

Figure \ref{fig:musig} shows the parameters $\mu,\sigma$, Figure \ref{fig:g} shows the derivatives of the transformation $g$, and Figure \ref{fig:tilde} shows the parameters $\tilde \mu, \tilde\sigma$ of the transformed SDE.

\begin{figure}
\begin{center}
\includegraphics[scale=0.5]{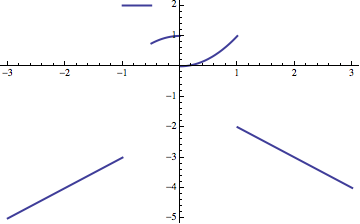}
\includegraphics[scale=0.5]{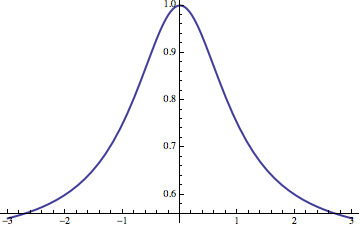}
\caption{The functions $\mu$ and $\sigma$.}\label{fig:musig}
\end{center}
\end{figure}

\begin{figure}
\begin{center}
\includegraphics[scale=0.4]{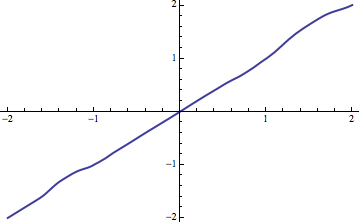}
\includegraphics[scale=0.4]{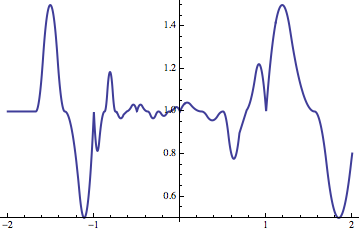}
\includegraphics[scale=0.4]{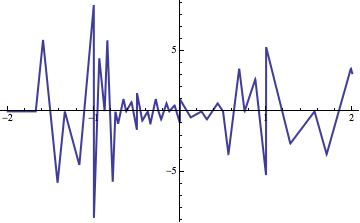}
\caption{The functions $g$, $g'$, and $g''$.}\label{fig:g}
\end{center}
\end{figure}

\begin{figure}
\begin{center}
\includegraphics[scale=0.6]{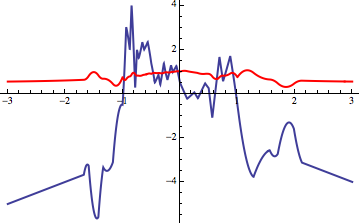}
\caption{The transformed parameters $\tilde \mu$ (blue) and $\tilde \sigma$ (red).}\label{fig:tilde}
\end{center}
\end{figure}

Again, we choose different values of $\varkappa=1/16,1/64,1/256$.
Figure \ref{fig:fancyplot} shows the estimated $\cL^2$-error between two consecutive discretizations of the transformed Euler-Maruyama method (EMT $1/\varkappa$)
in comparison to crude Euler-Maruyama (EM).

\begin{figure}
\begin{center}
\includegraphics[scale=0.7]{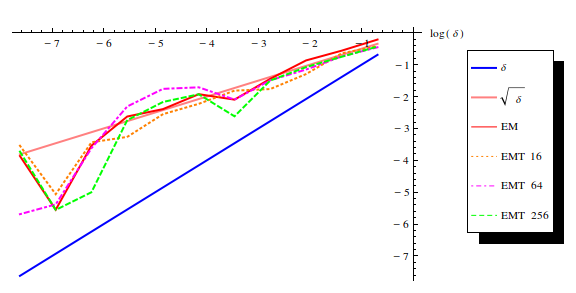}
\caption{The estimated $\cL^2$-error for different values of $\varkappa$ for Example \ref{ex:ex2}.}\label{fig:fancyplot}
\end{center}
\end{figure}

In this example we reach the calculated order $1/2$.
Whether the transformed method is better than Euler-Maruyama depends on the choice of $\varkappa$.

\end{example}

As the last example we solve an SDE appearing in insurance mathematics.

\begin{example}[Threshold dividend strategy]\label{ex:ex3}
In \cite{asmussen1997} the authors study the dividend maximization problem from risk theory in a diffusion model. They find that the optimal dividend policy is of threshold type with constant threshold level $b$.
This means that dividends should be paid at the maximum rate whenever the surplus process $X$ of the insurance company exceeds $b$. Otherwise, no dividends should be paid.
Our numerical scheme enables us to simulate the optimally controlled surplus process from \cite{asmussen1997}:
\[
dX_t=\left(\theta-K \, 1_{\lbrace x \ge b\rbrace}\right) dt + \sigma d W_t\,,
\]
where $\theta$ is the drift of the uncontrolled process, and $K$ is the maximum dividend rate.
We choose $\theta=1, K=1.8, \sigma=1$, and $b\approx 0.895635$ is the optimal threshold level, which can be calculated as in \cite{asmussen1997}.

Figure \ref{fig:barrplot} shows the estimated $\cL^2$-error between two consecutive discretizations of the transformed Euler-Maruyama method (EMT $1/\varkappa$)
in comparison to crude Euler-Maruyama (EM) for $\varkappa=1/16,1/64,1/256$.

\begin{figure}
\begin{center}
\includegraphics[scale=0.7]{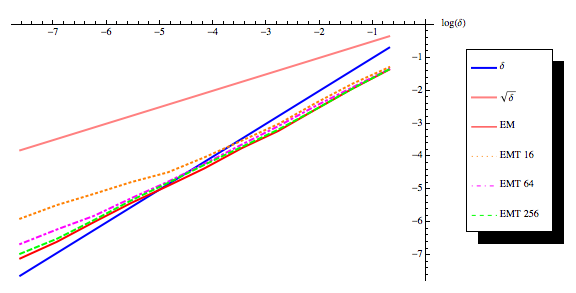}
\caption{The estimated $\cL^2$-error for different values of $\varkappa$ for Example \ref{ex:ex3}.}\label{fig:barrplot}
\end{center}
\end{figure}

\end{example}

Threshold type control strategies appear frequently when solving stochastic optimization problems in various fields of applied mathematics. Therefore, our scheme potentially serves for a wide range of applications.

\end{document}